\documentclass[
	paper=a4,%
	pagesize,%
	numbers=noendperiod,%
	captions=nooneline,%
	abstracton%
]{scrartcl}

\newcommand*{\mytitle}{Sibuya\ copulas} 
\newcommand*{\myauthorone}{Marius\ Hofert}
\newcommand*{\mycontactone}{Department\ of\ Mathematics,\ ETH\ Zurich,\ 8092\ Zurich,\ Switzerland,\ \href{mailto:marius.hofert@math.ethz.ch}{\nolinkurl{marius.hofert@math.ethz.ch}}}
\newcommand*{\mycomment}{The\ author\ (Willis\ Research\ Fellow)\ thanks\ Willis\ Re\ for\ financial\ support\ while\ this\ work\ was\ completed.}
\newcommand*{\myauthortwo}{Frederic\ Vrins}
\newcommand*{\mycontacttwo}{ING Belgium SA,\ Brussels,\ \href{mailto:frederic.vrins@ing.be}{\nolinkurl{frederic.vrins@ing.be}}}
\newcommand*{\mysubject}{Article}

\usepackage[T1]{fontenc}%
\usepackage{lmodern}%
\usepackage[american]{babel}%
\usepackage{microtype}%
\usepackage[nouppercase]{scrpage2}%
\usepackage{amsmath}%
\usepackage{mathtools}%
\usepackage{amssymb}%
\usepackage{amsthm}%
\usepackage{bm}%
\usepackage{bbm}%
\usepackage[round]{natbib} %
\usepackage{enumitem}%
\usepackage{graphicx}%
\usepackage{grffile}%
\usepackage{tabularx}%
\usepackage{dcolumn}%
\usepackage{booktabs}%
\usepackage{multirow}%
\usepackage[
	pdftex,%
	hypertexnames=false,%
	setpagesize=false,%
	pdfborder={0 0 0},%
   pdfstartview=Fit,%
	bookmarksopen=true,%
	bookmarksnumbered=true,%
	pdfauthor=\myauthorone,%
	pdftitle=\mytitle,%
	pdfsubject=\mysubject%
]{hyperref}%
\usepackage{breakurl}

\makeatletter
\newcommand\myisodate{\number\year-\ifcase\month\or 01\or 02\or 03\or 04\or 05\or 06\or 07\or 08\or 09\or 10\or 11\or 12\fi-\ifcase\day\or 01\or 02\or 03\or 04\or 05\or 06\or 07\or 08\or 09\or 10\or 11\or 12\or 13\or 14\or 15\or 16\or 17\or 18\or 19\or 20\or 21\or 22\or 23\or 24\or 25\or 26\or 27\or 28\or 29\or 30\or 31\fi}%
\makeatother
\setkomafont{pageheadfoot}{\normalfont}%
\automark{section}%
\setkomafont{captionlabel}{\sffamily\bfseries}%
\newcolumntype{d}[2]{D{.}{.}{#1.#2}}%
\newcommand*{\abstractnoindent}{}%
\let\abstractnoindent\abstract
\renewcommand*{\abstract}{\let\quotation\quote\let\endquotation\endquote
\abstractnoindent}
\deffootnote[1em]{1em}{0em}{\textsuperscript{\thefootnotemark}}%
\makeatletter%
\newcommand{\interitemtext}[1]{%
\begin{list}{}
{\itemindent=0mm\labelsep=0mm
\labelwidth=0mm\leftmargin=0mm
\addtolength{\leftmargin}{-\@totalleftmargin}}
\item #1
\end{list}}
\makeatother

\newtheoremstyle{mythmstyle}%
	{0.5em}%
	{0.5em}%
	{}%
	{}%
	{\sffamily\bfseries}%
	{}%
	{\newline}%
	{\thmname{#1}\ \thmnumber{#2}\ \thmnote{(#3)}}%
\newcommand*{\myskipalgo}{~\vspace{-1.5em}}%
\theoremstyle{mythmstyle}
\newtheorem{definition}{Definition}[section]%
\newtheorem{proposition}[definition]{Proposition}
\newtheorem{lemma}[definition]{Lemma}
\newtheorem{theorem}[definition]{Theorem}
\newtheorem{corollary}[definition]{Corollary}
\newtheorem{remark}[definition]{Remark}
\newtheorem{example}[definition]{Example}
\newtheorem{algorithm}[definition]{Algorithm}
\renewcommand*\proofname{Proof}
\makeatletter%
\renewenvironment{proof}[1][\proofname]{\par
  \pushQED{\qed}%
  \normalfont\topsep2\p@\@plus2\p@\relax
  \trivlist
  \item[\hskip\labelsep
     \sffamily\bfseries #1]\mbox{}\hfill\\*\ignorespaces
}{%
  \popQED\endtrivlist\@endpefalse
}
\makeatother

\newcommand*{\eps}{\varepsilon}
\newcommand{\T}{\ensuremath{^\mathsf{T}}\hspace{-0.5mm}}
\newcommand*{\omu}[3]{\underset{\mathclap{#3}}{\overset{\mathclap{#1}}{#2}}}
\newcommand*{\I}{\mathbbm{1}}
\newcommand*{\IN}{\mathbb{N}}

\newcommand*{\IR}{\mathbb{R}}

\newcommand*{\IP}{\mathbb{P}}
\newcommand*{\IE}{\mathbb{E}}

\newcommand*{\vt}{\vartheta}

\newcommand*{\FTD}{\operatorname*{FTD}}

\newcommand*{\Exp}{\operatorname*{Exp}}
\newcommand*{\Poi}{\operatorname*{Poi}}

\newcommand*{\U}{\operatorname*{U}}

\newcommand*{\Sinv}[3]{{S_{#1}^-(u_{#2}^{\phantom{-}}\hspace{#3})}}

\newcommand*{\textcite}[2][]{\citet[#1]{#2}}

\begin{document}
\thispagestyle{plain}
	\begin{center}
		{\sffamily\bfseries\LARGE\mytitle\par}
		\bigskip
		{\Large\myauthorone\footnote{\mycontactone. \mycomment},\ \myauthortwo\footnote{\mycontacttwo}\par
	    \bigskip
	    \myisodate\par}
	\end{center}
	\par\bigskip
	\begin{abstract}
		The standard intensity-based approach for modeling defaults is generalized by making the deterministic term structure of the survival probability stochastic via a common jump process. The survival copula of the vector of default times is derived and it is shown to be explicit and of the functional form as dealt with in the work of Sibuya. Besides the parameters of the jump process, the marginal survival functions of the default times appear in the copula. Sibuya copulas therefore allow for functional parameters and asymmetries. Due to the jump process in the construction, they allow for a singular component. Depending on the parameters, they may also be extreme-value copulas or L\'evy-frailty copulas. Further, Sibuya copulas are easy to sample in any dimension. Properties of Sibuya copulas including positive lower orthant dependence, tail dependence, and extremal dependence are investigated. An application to pricing first-to-default contracts is outlined and further generalizations of this copula class are addressed.
	\end{abstract}
	\minisec{Keywords}
		Sibuya type distributions, Marshall-Olkin copulas, extreme-value copulas, L\'evy-frailty copulas, default modeling, jump processes.
	\minisec{MSC2010} 
 	60E05, 62H99, 60G99, 62H05, 62H20.%
	\section{Introduction}
		A $d$-dimensional \textit{copula} is a $d$-dimensional distribution function with standard uniform univariate margins. The main goal of the present work is to construct a flexible class of $d$-dimensional copulas based on a multivariate default model and investigate its properties. The multivariate default model considered is a generalization of the standard intensity-based approach by using a common jump process. After presenting this model, the survival copula combining the default times is explicitly derived. It is of the functional form as appearing in \textcite{sibuya1959}. The resulting class of ``Sibuya copulas'' is remarkable in many ways. A first distinguishing feature compared to other copula classes is the fact that Sibuya copulas have functional parameters and therefore quite flexible in terms of its properties. As a second important feature, Sibuya copulas are not restricted to functional symmetry, i.e., exchangeability. This is important in large dimensions when exchangeability becomes a more and more restrictive assumption. Third, due to the construction with a common jump process, Sibuya copulas may have a singular component. This is important for applications such as multivariate default models since it translates to a positive probability that several components of a system default at the same time. Another important feature of large-dimensional copulas in applications is sampling. Due to the construction of Sibuya copulas via a default model, it is easy to draw vectors of random variates from these copulas. 
		\par
		Due to their general functional form, it seems difficult to determine the properties of Sibuya copulas in general. For investigating the properties, we therefore consider a working example throughout the paper. Even for this example, Sibuya copulas are seen to allow for many interesting properties. For example, the considered example is an extreme-value copula and also, as a special case, a L\'evy-frailty copula. Further, the lower and upper extremal-dependence coefficients may be derived explicitly. Moreover, for the bivariate case, the working example is of Marshall-Olkin type. 
		\par
		As an application, we consider the valuation of a financial derivative contract, known as first-to-default swaps. We derive the pricing equation in the proposed dependence framework, and obtain an analytical formula for the fair spread of such type of contracts. Finally, possible extensions of the construction principle for Sibuya copulas are given.
		\par
		The article is organized as follows. Section \ref{sec.int} recalls the standard intensity-based approach for default modeling. A generalization utilizing a common jump process is given in Section \ref{sec.int.gen}. In Section \ref{sec.joint} we compute the joint survival function of the default times. The corresponding Sibuya type copulas are then derived in Section \ref{sec.cop}. In Section \ref{sec.prop} the properties of this distributional class are investigated, including positive lower orthant dependence, tail dependence, extremal dependence, and sampling. Section \ref{sec.app} outlines an application to the pricing of first-to-default contracts, and Section \ref{sec.gen} briefly addresses possible extensions of the construction principle for Sibuya copulas. Finally, Section \ref{sec.concl} concludes. For the reader's convenience, proofs are given in the Appendix.
	\section{A default model and its implied dependence structure}
		In this section, we derive a dependence structure based on a default model in a stochastic intensity framework. First, the standard intensity-based model is reviewed. Next, the natural extension consisting of making the intensities stochastic is considered. The stochastic intensities are restricted to take the form of a non-decreasing deterministic part and a non-decreasing jump process. The latter, being common to all entities, generates the dependence among the constituents.
		\subsection{The standard intensity-based default}\label{sec.int}
			In the standard intensity-based approach, default times are modeled as the first-jump times of a (possibly non-homogeneous) Poisson process. In other words, the default time of the $i$th of $d$ components in a portfolio is modeled by a deterministic, non-negative \textit{intensity} $\tilde{\lambda}_i:[0,\infty)\to[0,\infty)$, $i\in\{1,\dots,d\}$. Given the intensity $\tilde{\lambda}_i$, the \textit{survival probability} $\tilde{p}_i(t)$ of component $i$ until time $t$ is given by
		  \begin{align}
		    \tilde{p}_i(t):=\exp(-\tilde{\Lambda}_i(t)),\ t\in[0,\infty),\label{ptilde}
		  \end{align}
		  for the \textit{integrated rate function} $\tilde{\Lambda}_i(t):=\int_0^t\tilde{\lambda}_i(s)\,ds$, $t\in[0,\infty)$. The canonical construction of the \textit{default time} $\tilde{\tau}_i$ is then given by
		  \begin{align}
		    \tilde{\tau}_i:=\inf\{t\ge0:\tilde{p}_i(t)\le U_i\},\label{tautilde}
		  \end{align}
		  for \textit{triggers} $U_i\sim\U[0,1]$, $i\in\{1,\dots,d\}$, see \textcite[pp.\ 227]{bieleckirutkowski2002} or \textcite[p.\ 122]{schoenbucher2003}. The \textit{survival function} $\tilde{S}_i$ for the $i$th component at time $t$ can now be computed as
		  \begin{align*}
		    \tilde{S}_i(t):=\IP(\tilde{\tau}_i>t)=\IP(\tilde{p}_i(t)\ge U_i)=\tilde{p}_i(t),\ t\in[0,\infty).
		  \end{align*}
		  \par
		  Usually, one seeks for generating default times that are mutually dependent. This is already introduced at various points in the literature, including, e.g., \textcite{li2000}, \textcite{schoenbucherschubert2001}, or \textcite{hofertscherer2010} who introduce dependence by assuming a joint model for the vector of trigger variables $(U_1,\dots,U_d)\T$. The following section we take another approach. We assume the triggers to be independent (see Section \ref{sec.gen} for possible extensions) and introduce dependence by a common jump process.
		\subsection{A generalized default model}\label{sec.int.gen}
			In what follows we generalize Mechanism (\ref{ptilde}) for modeling the term structure of survival probability of an entity by considering Cox processes instead of Poisson processes. The deterministic \textit{survival process} $(\tilde{p}_i(t))_{t\in[0,\infty)}$ is replaced by a stochastic one, i.e.,
		\begin{align}
	    p_i(t):=\exp(-X_{i,t}),\ t\in[0,\infty),\label{p}
	  \end{align}
		where $(X_{i,t})_{t\in[0,\infty)}$, $i\in\{1,\dots,d\}$, are right-continuous, increasing stochastic processes with independent increments and $X_{i,0}=0$, $i\in\{1,\dots,d\}$ that we naturally refer to as the \textit{integrated intensity process} (``IIP''). The idea underlying the stochastic extension (\ref{p}) of (\ref{ptilde}) is that samples of the default time may be generated by a more general model without altering the distribution.
		\par	
		As a realistic model for the processes $(X_{i,t})_{t\in[0,\infty)}$, $i\in\{1,\dots,d\}$, we consider
		\begin{align}
	    	X_{i,t}:=M_i(t)+J_t,\ t\in[0,\infty),\ i\in\{1,\dots,d\},\label{X}
		\end{align}
		where $M_i(t):=\int_0^t\mu_i(s)\,ds$, $t\in[0,\infty)$, with a deterministic function $\mu_i:[0,\infty)\to[0,\infty)$, analogous to $\tilde{\lambda}_i$ before, and $(J_t)_{t\in[0,\infty)}$ is a right-continuous, increasing jump process with independent increments and $J_0=0$.
		\begin{remark}
		The reader may find quite restrictive to focus on IIP of the form given by (\ref{X}). However, it is a rather general form, as we now explain. In order for $X_{i,t}$ to be an IIP, it needs to be almost surely non-decreasing, otherwise the process $p_i(t)$ may not be a \textit{proper survival process}. Therefore, if we restrict $X_{i,t}$ to be right-continuous, with independent increments and stationary, then by definition $X_{i,t}$ is a L\'evy subordinator, which in turns, implies that any IIP satisfying these constraints are necessarily of the form $\mu_it+J_t$ where $\mu_i$ is a non-negative constant, see, e.g., \textcite[p.\ 88]{conttankov2004}. The class of IIP we consider in (\ref{X}) is even more general than L\'evy subordinators in the sense that they can be non-stationary via the (possibly non-linear) functions $M_i(t)$. This shows that the class of IIP defined by (\ref{X}) already covers an important part of the admissible processes.
	   \end{remark}
		The intuition behind (\ref{X}) is that the individual term structure of the survival probability, modeled via $M_i$, is hit by the jump process $(J_t)_{t\in[0,\infty)}$ which models common shocks affecting the components. The default of an entity $i$ is then modeled similar to (\ref{tautilde}) via 
		\begin{align}
		  \tau_i:=\inf\{t\ge0:p_i(t)\le U_i\}\label{tau}
		\end{align}
		for $U_i\sim\U[0,1]$, $i\in\{1,\dots,d\}$, where $\bm{U}$ is assumed to be independent of $(X_{i,t})_{t\in[0,\infty)}$, $i\in\{1,\dots,d\}$. 
		\par
		Note that the default times $\tau_i$, $i\in\{1,\dots,d\}$, are naturally dependent due to the common jump process $(J_t)_{t\in[0,\infty)}$. Our main goal is to investigate this dependence. We first focus on the case where $\bm{U}\sim\U[0,1]^d$, i.e., the dependence is solely induced by $(J_t)_{t\in[0,\infty)}$. An extension to nested or hierarchical dependence structures, as well as dependent triggers is discussed in Sections \ref{sec.hier.jump} and \ref{sec.dep.trig}, respectively.
		\par
		Since the dependence structure resulting from Construction (\ref{p}) is quite general, we consider the following working example throughout this article. 
		\begin{example}[Working example]\label{WE}
			As a working example (``E''), consider $(J_t)_{t\in[0,\infty)}$ to be of the form
			\begin{align*}
				J_t\omu{\text{(E)}}{=}{}HN_t,\ t\in[0,\infty),
			\end{align*}
			for a non-homogeneous Poisson process $(N_t)_{t\in[0,\infty)}$ with $N_0=0$ and $N_t\sim\Poi(\Lambda(t))$, where $\Lambda(t):=\int_0^t\lambda(s)\,ds$ for a deterministic, non-negative function $\lambda:[0,\infty)\to[0,\infty)$ such that $\Lambda(0)=\lim_{t\downarrow0}\Lambda(t)=0$, and where $H\ge0$ is a constant. This choice corresponds to the model of \textcite{hullwhite2008} where the jump size is set constant and equal to $H$.
		\end{example}	
	\subsection{The joint survival function}\label{sec.joint}
		In this section, we derive the \textit{joint survival function} associated to our default model, namely $S(\bm{t}):=\IP(\tau_1>t_1,\dots,\tau_d>t_d)$, $\bm{t}:=(t_1,\dots,t_d)\T\in[0,\infty]^d$. An analytical expression of this survival distribution can be found in \textcite{vrins2010} in the particular case defined by the working example. In what follows, we need the following lemma.
	 	\begin{lemma}\label{lemma}
			Let $a_k\in[0,\infty)$, $k\in\{0,\dots,d\}$, $d\in\IN$, with $a_0:=0$. Further, let $b_k\in\IR$, $k\in\{1,\dots,d\}$, such that $b_{k+1}=cb_k$ for $c\in\IR$ and $k\in\{1,\dots,d-1\}$. \begin{enumerate}[label=(\arabic*),leftmargin=*,align=left,itemsep=0mm,topsep=1mm]
	      \item\label{p.1} $\sum_{k=1}^da_k=\sum_{k=1}^d(d-k+1)(a_{(k)}-a_{(k-1)})$, where $a_{(k)}$ denotes the $k$th order statistic of $a_k$, $k\in\{1,\dots,d\}$, i.e., $a_{(1)}\le\dots\le a_{(d)}$.
	      \item\label{p.2} $\sum_{k=1}^d(1-b_k)(a_{(k)}-a_{(k-1)})=(1-cb_{d})a_{(d)}+(c-1)\sum_{k=1}^{d}b_ka_{(k)}$.
	      \end{enumerate}
	  \end{lemma}
		\begin{proof}
			The proof is given in \ref{app.lemma}.
		\end{proof}
		The following theorem presents the joint survival function of the default times $\tau_i$, $i\in\{1,\dots,d\}$. Here, $\psi_Y(x)$ denotes the Laplace-Stieltjes transform of the random variable $Y$ at $x$, i.e., $\psi_Y(x):=\IE[\exp(-xY)]$, $x\in[0,\infty]$.
		\begin{theorem}\label{main.theorem}
			For $i\in\{1,\dots,d\}$, let the survival process $p_i$ and the corresponding default time $\tau_i$ be specified as in (\ref{p}), (\ref{X}), and (\ref{tau}), i.e., let 
			\begin{align*}
				\tau_i=\inf\{t\ge0:\exp(-(M_i(t)+J_t))\le U_i\},\ i\in\{1,\dots,d\}.
			\end{align*}
			Then, the joint survival function of the default times is given by
			\begin{align}
				S(\bm{t})=\prod_{i=1}^d\frac{\psi_{J_{t_{(i)}}-J_{t_{(i-1)}}}(d-i+1)}{\psi_{J_{t_{(i)}}}(1)}S_i(t_i),\label{S}
			\end{align}
			where $t_0:=0$ and $S_i(t):=\IP(\tau_i>t)$, $t\in[0,\infty]$, $i\in\{1,\dots,d\}$, denote the \textit{marginal survival functions} corresponding to $S$, given by $S_i(t)=\exp(-M_i(t))\psi_{J_t}(1)$, $t\in[0,\infty]$, $i\in\{1,\dots,d\}$.
		\end{theorem}
		\begin{proof}
			The proof is given in \ref{app.main}.
		\end{proof}
		We may infer from (\ref{S}) that the joint survival function $S$ is of a form as dealt with in \textcite{andersonlouisholmharvald1992}. It is parameterized by the corresponding marginal survival functions $S_i$, $i\in\{1,\dots,d\}$, and the common jump process $(J_t)_{t\in[0,\infty)}$. The following corollary presents the functional form of $S$ for the case of our working example. 
		\begin{corollary}[Working example]
			Let us consider the setup of the working example, i.e., Example \ref{WE}. In this case, the marginal survival functions are given by
			\begin{align*}
				S_i(t)\omu{\text{(E)}}{=}{}\exp\bigl(-(M_i(t)+\Lambda(t)(1-e^{-H}))\bigr).
			\end{align*}
			Since the non-homogeneous Poisson process $(N_t)_{t\in[0,\infty)}$ has increment distribution $N_{t_{(i)}}-N_{t_{(i-1)}}\sim\Poi(\Lambda(t_{(i)})-\Lambda(t_{(i-1)}))$ and since the Laplace-Stieltjes transform of $Y\sim\Poi(\gamma)$ equals $\psi_Y(x):=\exp\bigl(-\gamma(1-\exp(-x))\bigr)$, $x\in[0,\infty)$, $\gamma\in(0,\infty)$, we obtain
			\begin{align}
				S(\bm{t})%
				&\omu{\text{(E)}}{=}{}\exp\biggl(-\sum_{i=1}^d(1-e^{-H(d-i+1)})(\Lambda(t_{(i)})-\Lambda(t_{(i-1)}))+(1-e^{-H})\sum_{i=1}^d\Lambda(t_{(i)})\biggr)\notag\\
				&\phantom{\omu{\text{(E)}}{=}{}{}}\cdot\prod_{i=1}^dS_i(t_i).\label{WE.S}
			\end{align}
			By applying Lemma \ref{lemma} \ref{p.1} with $a_i:=\Lambda(t_i)$, $i\in\{1,\dots,d\}$, (\ref{WE.S}) can be simplified to
			  \begin{align}
			     S(\bm{t})\omu{\text{(E)}}{=}{}\prod_{i=1}^d\varphi(d-i+1,H,\Lambda(t_{(i)})-\Lambda(t_{(i-1)}))S_i(t_i),\label{WE.S.1}
			  \end{align}
			  where $\varphi(x,y,z):=\exp\bigl(-z(1-e^{-xy}-x(1-e^{-y}))\bigr)$, $x,y,z\in[0,\infty)$ denotes the \textit{jointure function} as introduced in \textcite{vrins2010}; note that the function $(1-e^{-xy}-x(1-e^{-y}))$ is non-positive for $x\in[1,\infty)$, $y\in[0,\infty)$, and it is decreasing in $y$ for any fixed $x\in[1,\infty)$. Further, by applying Lemma \ref{lemma} \ref{p.2} with $c:=e^H$, $a_i:=\Lambda(t_i)$, and $b_i:=e^{-H(d-i+1)}$, $i\in\{1,\dots,d\}$, the joint survival function in (\ref{WE.S}) can also be expressed as
			  \begin{align}
			      S(\bm{t})\omu{\text{(E)}}{=}{}\exp\biggl((1-e^{-H})\sum_{i=1}^{d}(1-e^{-H(d-i)})\Lambda(t_{(i)})\biggr)\prod_{i=1}^dS_i(t_i).\label{WE.S.2}
			  \end{align}
		\end{corollary}
		\subsection{The implied dependence structure}\label{sec.cop}
			With the joint survival function and the corresponding marginal survival functions at hand, one can derive the copula which provides a link between these two pieces of the multivariate default model. 
		\begin{corollary}
			Let $\bm{u}:=(u_1,\dots,u_d)\T\in[0,1]^d$ and let $S_i$, $i\in\{1,\dots,d\}$, be given as in Theorem \ref{main.theorem}. Further, let $S_i^-$ denote the generalized inverse corresponding to $S_i$, $i\in\{1,\dots,d\}$, and let $\Sinv{\cdot\,}{\cdot\,}{-1.3mm}_{(i)}$ denote the $i$th order statistic of $S_i^-(u_i)$, $i\in\{1,\dots,d\}$, i.e., $\Sinv{\cdot\,}{\cdot\,}{-1.3mm}_{(1)}\le\dots\le\Sinv{\cdot\,}{\cdot\,}{-1.3mm}_{(d)}$. The copula $C$ corresponding to the joint survival function (\ref{S}) is then given by
		  \begin{align}
		    C(\bm{u})=S(S_1^-(u_1),\dots,S_d^-(u_d))=\prod_{i=1}^d\frac{\psi_{J_{\Sinv{\cdot\,}{\cdot\,}{-1.3mm}_{(i)}}-J_{\Sinv{\cdot\,}{\cdot\,}{-1.3mm}_{(i-1)}}}(d-i+1)}{\psi_{J_{\Sinv{\cdot\,}{\cdot\,}{-1.3mm}_{(i)}}}(1)}u_i.\label{C} 
		  \end{align}
		  	Since $S_i^-(t)\le S_j^-(t)$, $t\in[0,\infty)$, if and only if $S_i(t)\le S_j(t)$, $t\in[0,\infty)$, the diagonal corresponding to (\ref{C}) is given by
		 \begin{align*}
		    C(u,\dots,u)=u^d\prod_{i=1}^d\frac{\psi_{J_{S_{(i)}^-(u)}-J_{S_{(i-1)}^-(u)}}(d-i+1)}{\psi_{J_{S_{(i)}^-(u)}}(1)}.
		 \end{align*}
		\end{corollary}
	  \begin{remark}
	  		The copula $C$ is in fact the survival copula of the vector of default times $(\tau_1,\dots,\tau_d)\T$. From the form above one recognizes that $C$ is of the form of a class of distributions introduced by \textcite{sibuya1959}, who called 
	  \begin{align*}
	    		\prod_{i=1}^d\frac{\psi_{J_{\Sinv{\cdot\,}{\cdot\,}{-1.3mm}_{(i)}}-J_{\Sinv{\cdot\,}{\cdot\,}{-1.3mm}_{(i-1)}}}(d-i+1)}{\psi_{J_{\Sinv{\cdot\,}{\cdot\,}{-1.3mm}_{(i)}}}(1)} 
	  \end{align*}
		the \textit{dependence function} of the joint survival function $S(\bm{t})$. We therefore refer to class of copulas $C$ as given in (\ref{C}) as \textit{Sibuya copulas}.
	   \end{remark}
		\par
		The idea of constructing a copula via a multivariate default model was recently applied by \textcite{maischerer2009a}. In their work, $(X_{i,t})_{t\in[0,\infty)}$ takes the form $(\Lambda_{h_i(t)})_{t\in[0,\infty)}$, $i\in\{1,\dots,d\}$, for a common L\'evy subordinator $(\Lambda_t)_{t\in[0,\infty)}$ combined with a rescaling of the time-clock via functions $h_i$, $i\in\{1,\dots,d\}$. This results in a non-stationary, non-decreasing stochastic process $(X_{i,t})_{t\in[0,\infty)}$. The rescalings $h_i$ are monotonically increasing entity-dependent functions derived from the riskiness of the entities, i.e., the subordinator time is passing more rapidly for riskier entities so that the survival process at some standard time point is lower for riskier entities than for safer ones. This approach also results in a tractable dependence model for defaults. However, the derived copula is restricted to functional symmetry, also known as exchangeability. This drawback is shared by many copula classes including the class of Archimedean copulas and also, partly, by nested Archimedean copulas. It becomes a more and more restrictive assumption in large dimensions since it implies that, e.g., all bivariate margins of the copula are equal. This, in turn, implies that, e.g., it is not possible to construct asymmetric (non-exchangeable) joint distributions if all margins are identical. Such restrictive properties are rarely observed, especially for large-dimensional portfolios. Sibuya copulas do not suffer from this drawback. 
		\begin{example}[Working example]
			From (\ref{WE.S.1}) and (\ref{WE.S.2}) we may infer that the Sibuya copula $C$ in the setup of the working example is given by
			\begin{align*}
				C(\bm{u})&\omu{\text{(E)}}{=}{}\prod_{i=1}^d\varphi\bigl(d-i+1,H,\Lambda(\Sinv{\cdot\,}{\cdot\,}{-1.3mm}_{(i)})-\Lambda(\Sinv{\cdot\,}{\cdot\,}{-1.3mm}_{(i-1)})\bigr)\prod_{i=1}^du_i\\
				&=\exp\biggl((1-e^{-H})\sum_{i=1}^{d}(1-e^{-H(d-i)})\Lambda(\Sinv{\cdot\,}{\cdot\,}{-1.3mm}_{(i)})\biggr)\prod_{i=1}^du_i.
			\end{align*}
			If $H=0$, then $C(\bm{u})=\Pi(\bm{u}):=\prod_{i=1}^d u_i$ is the \textit{independence copula}. In the bivariate case, $C$ can be written as
			\begin{align}
			      C(u_1,u_2)\omu{\text{(E)}}{=}{}&\exp((1-e^{-H})^2\Lambda(\Sinv{\cdot\,}{\cdot\,}{-1.3mm}_{(1)}))u_1u_2\notag\\
			=&\exp((1-e^{-H})^2\Lambda(\min\{S_1^-(u_1),S_2^-(u_2)\}))u_1u_2\label{C.2d}
			\end{align}
			From this functional form one may derive that $C$ allows for a singular component, given by $\{\bm{u}\in[0,1]^2\,\vert\,S_1^-(u_1)=S_2^-(u_2)\}$.
		\end{example}  
	\section{Properties of the copula}\label{sec.prop}
	    In this section we investigate some properties of the copula $C$. We present results about positive lower orthant dependence, tail dependence, extremal dependence, and sampling. Due to the quite general form of a Sibuya copula $C$, see (\ref{C}), it is difficult to investigate tail and extremal dependence, even under (E), i.e., for the case of the working example. In Section \ref{tail} and \ref{extreme}, we therefore work out the details under the additional assumption (``A''), which means $M_i(t)=\mu_it$, $i\in\{1,\dots,d\}$, and $\Lambda(t)=\lambda t$, $t\in[0,\infty)$, where $\mu_i$ and $\lambda$ are non-negative constants. 
		\par
		Let us first explore this case a bit. Under (A), the generalized inverses of the marginal survival functions $S_i$, $i\in\{1,\dots,d\}$, are given explicitly by $S_i^{-}(u)=-\log(u)/\lambda_i$, where we define $\lambda_i:=\mu_i+\lambda(1-e^{-H})$, $i\in\{1,\dots,d\}$, for convenience. In this setup, a Sibuya copula can be expressed as
		  \begin{align}
		    C(\bm{u})&\omu{\text{(A)}}{=}{}\prod_{i=1}^du_i\varphi\bigl(d-i+1,H,\log(u_{\cdot}^{-\lambda/\lambda_{\cdot}})_{(i)}-\log(u_{\cdot}^{-\lambda/\lambda_{\cdot}})_{(i-1)}\bigr)\notag\\
		    &=\prod_{i=1}^{d}u_i(u_{\cdot}^{-\lambda/\lambda_{\cdot}})_{(i)}^{(1-e^{-H})(1-e^{-H(d-i)})},\label{C.A}
		  \end{align}
		  with $(u_{\cdot}^{-\lambda/\lambda_{\cdot}})_{(0)}:=1$. The corresponding diagonal is given by
			  \begin{align}
			      C(u,\dots,u)&\omu{\text{(A)}}{=}{}u^d\prod_{i=1}^d\varphi(d-i+1,H,\log u^{-\lambda(1/\lambda_{[i]}-1/\lambda_{[i-1]})})\notag\\
			      &=u^{d-\lambda(1-e^{-H})\sum_{i=1}^{d}(1-e^{-H(d-i)})/\lambda_{[i]}},\label{C.A.diag}
			  \end{align}
			  i.e., a power function, where the subscript $\lambda_{[i]}$ stands for the $i$th largest of $\lambda_1,\dots,\lambda_d$ with $\lambda_{[0]}:=\infty$.
		\par
		The following remark addresses several properties of this copula.
		\begin{remark}
			\begin{enumerate}[label=(\arabic*),leftmargin=*,align=left,itemsep=0mm,topsep=1mm]
				\item Assuming $\mu_i\in(0,\infty)$, $i\in\{1,\dots,d\}$, if $H=0$ or $\lambda=0$, then $C$ is the independence copula $\Pi$. Further, if $H\uparrow\infty$ and $\lambda\uparrow\infty$, then $C$ becomes the upper Fr\'{e}chet bound copula $M$. We therefore conclude that Sibuya copulas allow to capture the full range of positive lower orthant dependence.
				\item The Sibuya copula $C$ as given in (\ref{C.A}) allows for asymmetries and therefore more realistic dependence structures, especially in large dimensions.
				\item Also note that this Sibuya copula is max-stable and therefore an extreme-value copula, see, e.g., \textcite[pp.\ 95]{nelsen2007}, hence Sibuya copulas can be extreme-value copulas.
				\item In the homogeneous case, i.e., $\mu_i=:\mu$ for all $i\in\{1,\dots,d\}$, (\ref{C.A}) can be written as
						\begin{align*}
							C(\bm{u})&=\prod_{i=1}^du_i(u_{[i]})^{-c(1-e^{-H(d-i)})}=\biggl(\,\prod_{i=1}^du_{(i)}\biggr)\biggl(\,\prod_{i=1}^d(u_{[i]})^{-c(1-e^{-H(d-i)})}\biggr)\\
							&=\prod_{i=1}^du_{(i)}^{1-c(1-e^{-H(d-i)})},
						\end{align*}
						where $c:=\lambda(1-e^{-H})/(\mu+\lambda(1-e^{-H}))$, and the last equation follows by changing the order of multiplication in the second product of the second last equation. Thus, under (A), homogeneous Sibuya copulas are L\'evy-frailty copulas, see \textcite{maischerer2009b}.
				\item In the bivariate case, (\ref{C.A}) becomes
					  \begin{align*}
					      C(u_1,u_2)\omu{\text{(A)}}{=}{}\min\{u_1^{-\lambda/\lambda_1},u_2^{-\lambda/\lambda_2}\}^{(1-e^{-H})^2}u_1u_2=\min\{u_1^{1-\vt_1}u_2,u_1u_2^{1-\vt_2}\},
					  \end{align*}
					  where $\vt_i:=(1-e^{-H})^2\lambda/\lambda_i\in[0,1]$, i.e. a Marshall-Olkin copula, see \textcite{marshallolkin1967}. Note that in this case, one has explicit formulas for Spearman's rho and Kendall's tau, see, e.g., \textcite{embrechtslindskogmcneil2001}, for the tail-dependence coefficients \textcite[p.\ 215]{nelsen2007}, as well as for the probability of falling on the singular component, see, e.g., \textcite[p.\ 54]{nelsen2007}.
			\end{enumerate}
		\end{remark}
	\subsection{Positive lower orthant dependence}
		It follows directly from Equation (\ref{C}) that $C(\bm{u})\ge\prod_{i=1}^du_i$, i.e., that $C$ is \textit{positive lower orthant dependent}, see \textcite[p.\ 21]{joe1997}. For the bivariate case this property is also called \textit{positive quadrant dependence} and it implies, by definition, that all measures of concordance such as Spearman's rho, Kendall's tau are greater than or equal to zero. By Equation (\ref{C}), the dependence function directly controls this dependence since
		\begin{align}
			\frac{\IP(U_1\le u_1,\dots,U_d\le u_d)}{\IP(U_1\le u_1)\cdots\IP(U_d\le u_d)}=\prod_{i=1}^d\frac{\psi_{J_{\Sinv{\cdot\,}{\cdot\,}{-1.3mm}_{(i)}}-J_{\Sinv{\cdot\,}{\cdot\,}{-1.3mm}_{(i-1)}}}(d-i+1)}{\psi_{J_{\Sinv{\cdot\,}{\cdot\,}{-1.3mm}_{(i)}}}(1)}\ge 1.\label{bad.news}
		\end{align}		
		Since the left-hand side of (\ref{bad.news}) can be written as
		\begin{align*}
			\frac{\IP(U_j\le u_j\,\vert\,U_1\le u_1,\dots,U_{j-1}\le u_{j-1},U_{j+1}\le u_{j+1},\dots,U_d\le u_d)}{\IP(U_j\le u_j)}
		\end{align*}
		for any $j\in\{1,\dots,d\}$, one also says that the copula $C$ has the ``bad news propagation'' effect.
	\subsection{Tail dependence}\label{tail}
	    We pursue with a bivariate notion of association known as tail dependence. For $X_i\sim F_i$, $i\in\{1,2\}$, with joint copula $C$, the \textit{lower} and \textit{upper tail-dependence coefficient} $\lambda_l$ and $\lambda_u$, respectively, are given by
	    \begin{align*}
	        \lambda_l&:=\lim_{u\downarrow0}\IP(X_2\le F_2^-(u)\,\vert\,X_1\le F_1^-(u))=\lim_{u\downarrow0}\frac{C(u,u)}{u},\\
	        \lambda_u&:=\lim_{u\uparrow1}\IP(X_2>F_2^-(u)\,\vert\,X_1>F_1^-(u))=\lim_{u\uparrow1}\frac{1-2u+C(u,u)}{1-u},
	    \end{align*}
	    where the limits are assumed to exist. By definition, the lower, respectively upper, tail-dependence coefficient tells us the likelihood, in the limit, that $X_1$ and $X_2$ are both small, respectively large, simultaneously. Note that if $(U_1,U_2)\T\sim C$, then $(1-U_1,1-U_2)\T\sim\hat{C}$, the survival copula corresponding to $C$. Therefore, the lower and upper tail-dependence coefficients interchange when going from $C$ to its survival copula $\hat{C}$.
	    \par
	    In our case, $C$ is the survival copula of the vector of default times $(\tau_1,\tau_2)\T$. Thus, the lower, respectively upper, tail-dependence coefficient tells us the likelihood, in the limit, that the two default times $\tau_1,\tau_2$ are jointly large, respectively small. So upper tail dependence means that a joint default model with a Sibuya copula as dependence structure produces joint defaults within a short amount of time.
	  	\par
		If one assumes only the setup of the working example, then one can at least say that
	 	\begin{align*}
	       \lambda_l\omu{\text{(E)}}{=}{}\lim_{u\downarrow0}u\exp((1-e^{-H})^2\Lambda(\min\{S_1^-(u),S_2^-(u)\})).
	    \end{align*}
	    Thus, if at least one of the individual survival functions $S_i(t)$ is zero at some finite $t$, then $\lambda_l=0$. Further, if $\Lambda$ is bounded, then $\lambda_l=0$. Under (A), $C$ is a Marshall-Olkin copula. Thus, the lower and upper tail-dependence coefficients are given by 
		\begin{align*}
			\lambda_l\omu{\text{(A)}}{=}{}0\ \text{and}\ \lambda_u\omu{\text{(A)}}{=}{}\min\{\vt_1,\vt_2\}=\frac{(1-e^{-H})^2\lambda}{\max\{\mu_1,\mu_2\}+\lambda(1-e^{-H})},
		\end{align*}
		respectively, see, e.g., \textcite[p.\ 215]{nelsen2007}. It follows that $H\downarrow0$ or $\lambda\downarrow0$ or $\max\{\mu_1,\mu_2\}\uparrow\infty$ implies $\lambda_u\downarrow0$. So if we suppose, in the limit, that entity $i\in\{1,2\}$ is extremely safe, i.e., $\lambda_i=\mu_i+\lambda(1-e^{-H})\downarrow0$, then also $\lambda_u\downarrow0$. Further, if $H\uparrow\infty$ and $\lambda\uparrow\infty$ or if $H\uparrow\infty$ and $\mu_i\downarrow0$, $i\in\{1,2\}$, then $\lambda_u\uparrow1$.
		\subsection{Extremal dependence}\label{extreme}
	    The notion of extremal dependence was introduced by \textcite{frahm2006}. For $X_i\sim F_i$, $i\in\{1,\dots,d\}$, with joint copula $C$, the \textit{lower} and \textit{upper extremal-dependence coefficient} $\lambda_l$ and $\lambda_u$, respectively, are given by
	    \begin{align*}
	        \eps_l:&=\lim_{u\downarrow0}\IP(\max\limits_{1\le i\le d}\{F_i(X_i)\}\le u\,\vert\min\limits_{1\le i\le d}\{F_i(X_i)\}\le u)\\
					&=\lim_{u\downarrow0}\frac{C(u,\dots,u)}{1-\hat{C}(1-u,\dots,1-u)},\\
	        \eps_u:&=\lim_{u\uparrow1}\IP(\min\limits_{1\le i\le d}\{F_i(X_i)\}>u\,\vert\max\limits_{1\le i\le d}\{F_i(X_i)\}>u)\\
				&=\lim_{u\uparrow1}\frac{\hat{C}(1-u,\dots,1-u)}{1-C(u,\dots,u)},
	    \end{align*}
	    where the limits are assumed to exist. By definition, the lower (upper) extremal-dependence coefficient tells us the likelihood, in the limit, that the largest (smallest) value of $F_i(X_i)$, $i\in\{1,\dots,d\}$, is small (large) given that the smallest (largest) value is. Applied to the setup where $C$ is the survival copula of the default times, this means that the lower (upper) extremal-dependence coefficient tells us the likelihood that the smallest (largest) default time is large (small) given that the largest (smallest) is. Thus, given that the first default happened within a short amount of time, the upper extremal-dependence coefficient tells us the likelihood of all other defaults also happening within a short amount of time. The following proposition gives explicit formulas for $\eps_l$ and $\eps_u$ under (A).
	    \begin{proposition}\label{prop.extreme}
	      Under (A), the lower and upper extremal-dependence coefficients are given by
	      \begin{align*}
	          \eps_l&\omu{\text{(A)}}{=}{}0,\\
	          \eps_u&\omu{\text{(A)}}{=}{}\sum_{I\subseteq\{1,\dots,d\}}(-1)^{\vert I\vert+1}\frac{\vert I\vert-\lambda(1-e^{-H})\sum_{i=\vert I\vert+1}^d(1-e^{-H(d-i)})/\lambda_{[i]}}{d-\lambda(1-e^{-H})\sum_{i=1}^d(1-e^{-H(d-i)})/\lambda_{[i]}},
	      \end{align*}  
	      respectively, where the sum extends over all $2^d$ subsets $I$ of $\{1,\dots,d\}$.
	    \end{proposition}
	 	\begin{proof}
			The proof is given in \ref{app.prop}.
		\end{proof}
	  \subsection{Sampling}
	  The intuitive construction principle of Sibuya copulas via a default model can be used for simulation. Principally, we have to simulate the vector $\bm{\tau}:=(\tau_1,\dots,\tau_d)\T$ of individual default times and then return $(S_1(\tau_1),\dots,S_d(\tau_d))\T$, a vector of random variates from the Sibuya copula $C$. Sampling a vector $\bm{\tau}$ involves drawing a vector $\bm{U}\sim\U[0,1]^d$ and sampling a path of the jump process $(J_t)_{t\in[0,T]}$, where $T$ is such that $p_i(T)\le U_i$ for all $i\in\{1,\dots,d\}$. Then, $\bm{\tau}$ is determined. Note that the number of occurrences to be sampled from the jump process depends on the given trigger variates $U_i$, $i\in\{1,\dots,d\}$, as well as on the deterministic functions $M_i$, $i\in\{1,\dots,d\}$. The following algorithm describes the general sampling procedure of $C$.
	  \begin{algorithm}\label{algo}
			\myskipalgo		
			\linespread{1.22}\normalfont
			\begin{tabbing}
				\hspace{8mm} \= \hspace{4mm} \= \hspace{4mm} \= \hspace{4mm} \= \hspace{12mm} \= \kill
				(1) \> \textbf{sample} $U_i\sim\U[0,1]$, $i\in\{1,\dots,d\}$\\
				(2) \> $t_{h,0}:=0$, $t_0:=0$, $k:=1$, and $I_k:=\{1,\dots,d\}$\\
				(3) \> \textbf{repeat} \{\\
				(4) \>\> \textbf{sample} the $k$th occurence $t_k$ of the jump process $(J_t)_{t\in[0,\infty)}$\\
				(5) \>\> \textbf{find} $I\subseteq I_k:i\in I\ \Leftrightarrow\ U_i\ge p_i(t_k)=\exp(-(M_i(t_k)+J_{t_k}))$\\
				(6) \>\> \textbf{for} $i\in I$ \{ \>\>\> \# find $\tau_i$ for all $i\in I$\\
				(7) \>\>\> \textbf{if} $\bigl(U_i\le p_i(t_k-)=\exp(-(M_i(t_k)+J_{t_{k-1}}))\bigr)$ $\tau_i:=t_k$\\ %
				(8) \>\>\> \textbf{else} \{ \>\> \# $U_i\in(p_i(t_k-),p_i(t_{k-1}))$\\ %
				(9) \>\>\>\> \textbf{find} $\tau_i$ on $(t_{k-1},t_k)$ via $\tau_i:=M_i^-(-\log U_i-J_{t_{k-1}})$\\
				(10) \>\>\> \}\\
				(11) \>\> \} \\
				(12) \>\> $I_{k+1}:=I_k\backslash I$ \>\>\> \# indices $i$ for which $\tau_i$ have not been determined yet\\
				(13) \>\> \textbf{if} ($I_{k+1}=\emptyset$) \textbf{break}\\
				(14) \>\> \textbf{else} $k:=k+1$\\
				(15) \> \}\\
				(16) \> \textbf{return} $(S_1(\tau_1),\dots,S_d(\tau_d))\T$
			\end{tabbing}
	  \end{algorithm}
	  \par
	  Under (E), $(J_t)_{t\in[0,\infty)}$ is given by $J_t=HN_t$, $t\in[0,\infty)$, for $N_t\sim\Poi(\Lambda(t))$. In this case we have $J_{t_k}=Hk$, $k\in\{1,2,\dots\}$. Further, Step (4) of Algorithm \ref{algo} can be achieved with the following algorithm, see, e.g., \textcite[p.\ 257]{devroye1986}.
		\begin{algorithm}\label{algo.Poi}
			\myskipalgo		
			\linespread{1.22}\normalfont
			\begin{tabbing}
				\hspace{8mm} \= \hspace{4mm} \= \hspace{26mm} \= \kill
				(1)	\>\textbf{sample} $E\sim\Exp(1)$\\
				(2)	\> $t_{h,k}:=t_{h,k-1}+E$ \>\> \# $k$th occurrence of a homogeneous Poisson\\
				(3) 	\>\>\> \# process with unit intensity\\
				(4)	\> $t_k:=\Lambda^-(t_{h,k})$ \>\> \# $k$th occurrence of a non-homogeneous Poisson\\
				(5) 	\>\>\> \# process with integrated rate function $\Lambda$\\
			\end{tabbing}
		\end{algorithm}
		Note that under (A), Step (9) of Algorithm \ref{algo} boils down to setting $\tau_i:=(-\log U_i-H(k-1))/\mu_i$ and Step (4) of Algorithm \ref{algo.Poi} to $t_k:=t_{h,k}/\lambda$.
	  \begin{example}\label{ex.sample}
			Let us consider the copula $C$ as given in (\ref{C.2d}). Since under (A), the bivariate $C$ is a Marshall-Olkin copula, we consider a more general example here. For this, let the ``intensities'' $\lambda$ and $\mu_i$, $i\in\{1,\dots,d\}$, be linear (instead of constant), i.e., let
			\begin{align*}
				\lambda(s)=a_\lambda s+b_\lambda,\quad\mu_i(s)=a_is+b_i,\ i\in\{1,\dots,d\},
			\end{align*}
			where $a_\lambda,b_\lambda,a_i,b_i\in[0,\infty)$, $i\in\{1,\dots,d\}$. Further, let us assume the non-trivial case where not both $a_\lambda$ ($a_i$) and $b_\lambda$ ($b_i$) are zero simultaneously. Letting $c:=(1-e^{-H})$ we obtain
			\begin{align*}
				\Lambda(t)&=a_\lambda t^2/2+b_\lambda t,\quad M_i(t)=a_it^2/2+b_it,\notag\\
				p_i(t)&=\exp(-(a_it^2/2+b_it+HN_t)),\notag\\
				S_i(t)&=\exp(-((a_i+ca_\lambda)t^2/2+(b_i+cb_\lambda)t))
			\end{align*}
			for $i\in\{1,\dots,d\}$. The corresponding inverses are
			\begin{align*}
				\Lambda^-(t)&=t/b_\lambda\I_{\{a_\lambda=0\}}+(\sqrt{b^2+2at}-b)/a\I_{\{a_\lambda>0\}},\\
				M_i^-(t)&=t/b_i\I_{\{a_i=0\}}+(\sqrt{b^2+2at}-b)/a\I_{\{a_i>0\}},\\
				S_i^-(u)&=\frac{\sqrt{(b_i+cb_\lambda)^2-2(a_i+ca_\lambda)\log(u)}-(b_i+cb_\lambda)}{a_i+ca_\lambda}.
			\end{align*}		
			Note that if $t_k$ denotes the $k$th occurrence of the non-homogeneous Poisson $(N_t)_{t\in[0,\infty)}$ process with integrated rate function $\Lambda$, then
			\begin{align*}
				p_i^-(t_k)&=\exp(-(a_it_k^2/2+b_it_k+Hk)),\\
				p_i^-(t_k-)&=\exp(-(a_it_k^2/2+b_it_k+H(k-1))).
			\end{align*}
		With these quantities one can evaluate the copula $C$ and apply Algorithm \ref{algo} for drawing vectors of random variates from $C$. Figures \ref{fig1} and \ref{fig2} show two examples. Note that although we choose linear intensities, the resulting structures are quite different, e.g., see the shape of the singular components. Further, one may again infer that these copulas are able to capture highly asymmetric structures.  
		\begin{figure}[htbp]
		 	\centering
		   \includegraphics[width=0.52\textwidth]{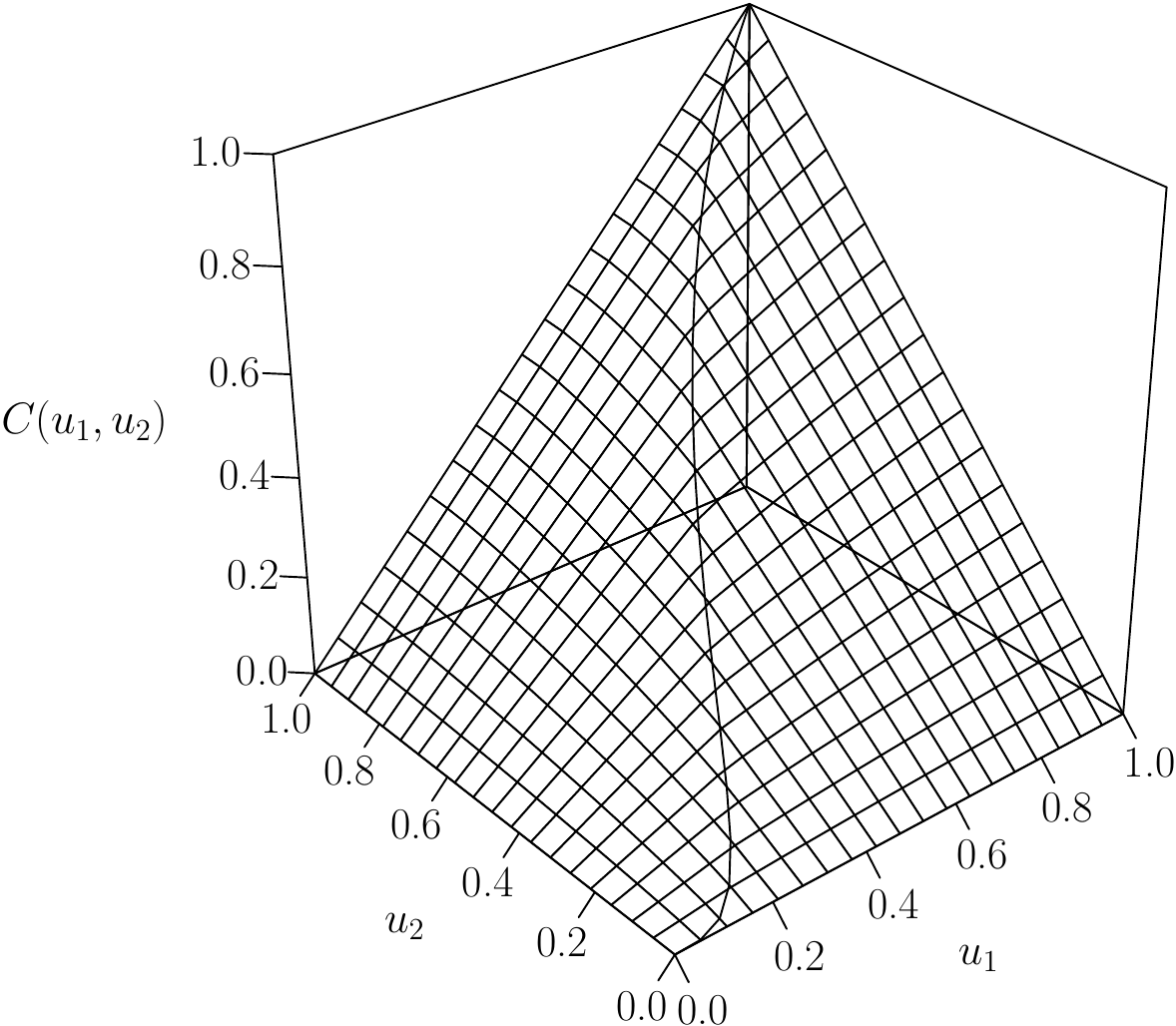}%
		   \hfill
		 	\includegraphics[width=0.43\textwidth]{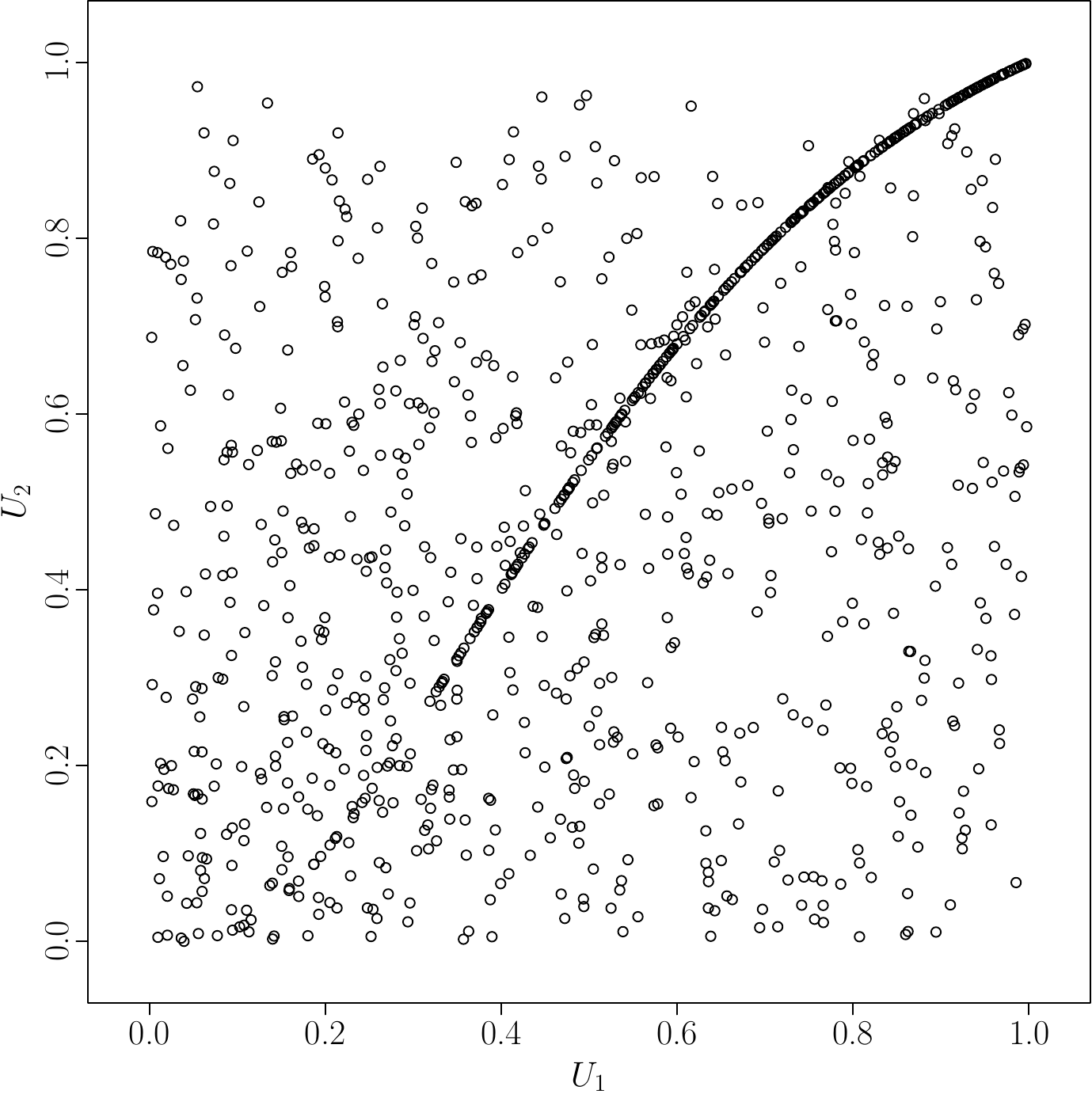}%
		 	\setcapwidth{\textwidth}%
		 	\caption{Sibuya Copula (\ref{C.2d}) and singular component for Example (\ref{ex.sample}) with $H=10$, $a_\lambda=0.1$, $b_\lambda=4$, $\bm{a}=(a_1,a_2)\T=(1,100)\T$, and $\bm{b}=(b_1,b_2)\T=(5,0)\T$ (left). $1\,000$ generated vectors of random variates from this copula (right). The lower and upper tail-dependence coefficients are given by $\lambda_l=0$ and $\lambda_u=0.44$, respectively.} 
		 	\label{fig1}
		\end{figure}
		\begin{figure}[htbp]
		 	\centering
		   \includegraphics[width=0.52\textwidth]{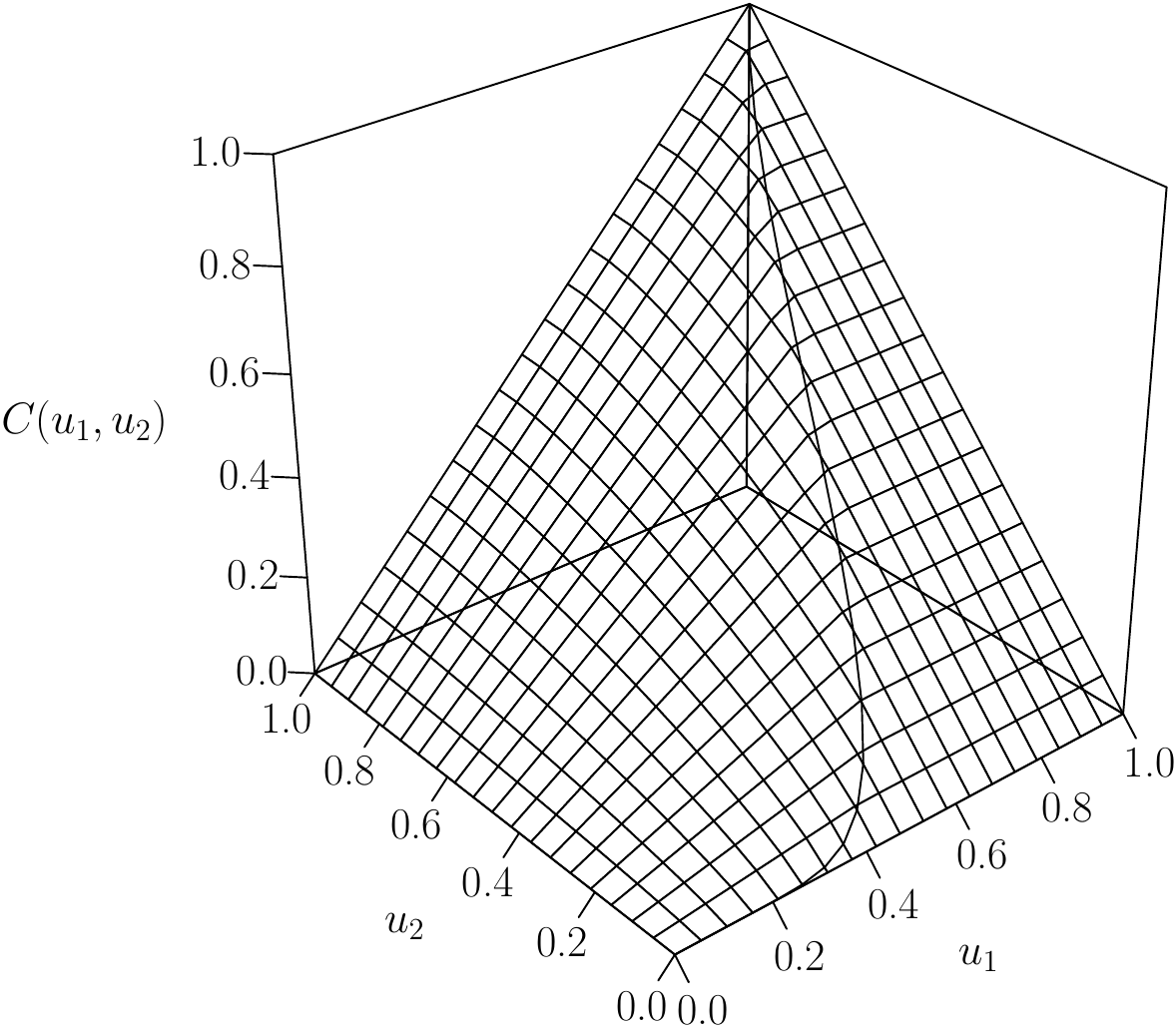}%
		   \hfill
		 	\includegraphics[width=0.43\textwidth]{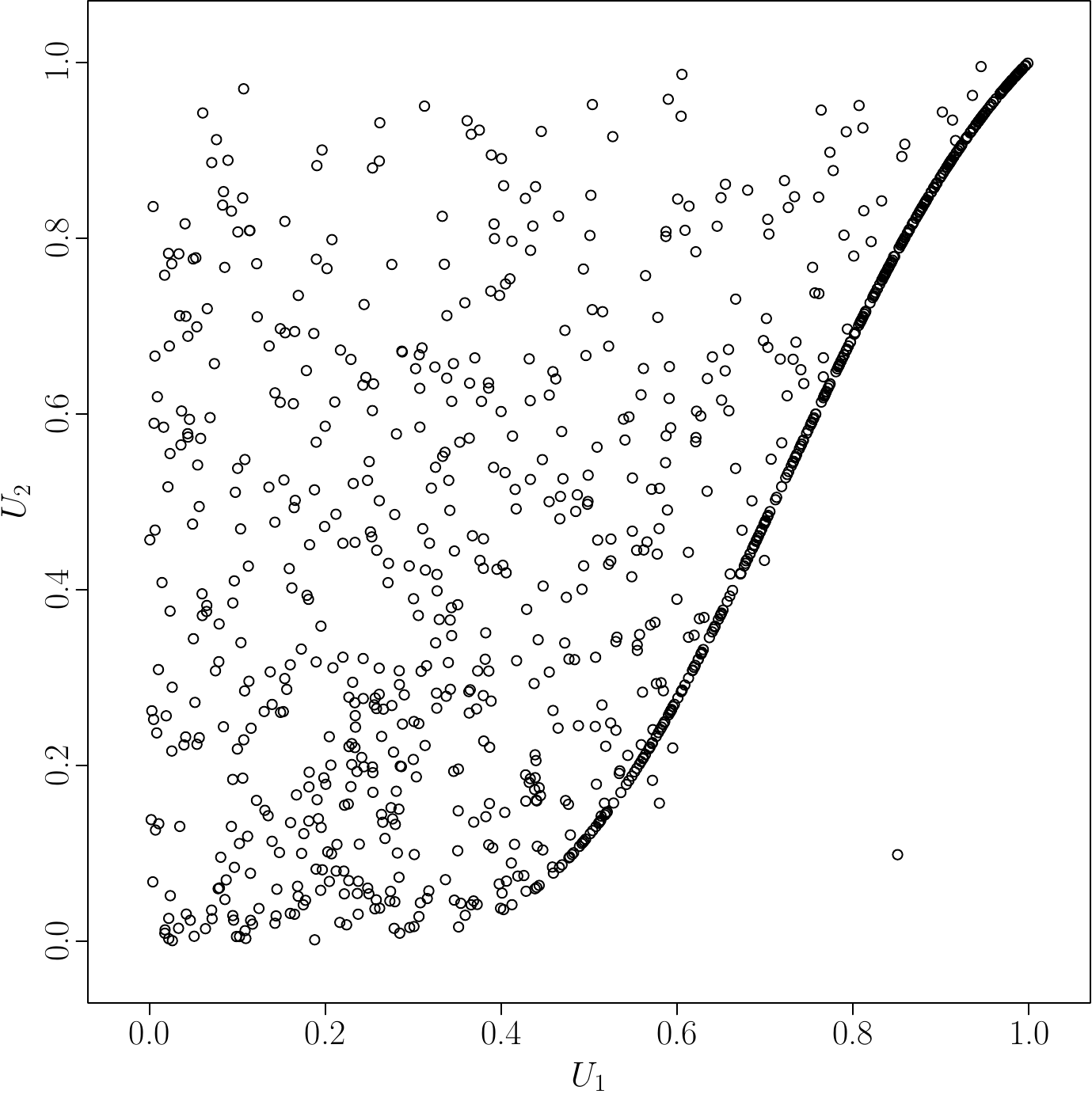}%
		 	\setcapwidth{\textwidth}%
		 	\caption{Sibuya Copula (\ref{C.2d}) and singular component for Example (\ref{ex.sample}) with $H=10$, $a_\lambda=0.1$, $b_\lambda=4$, $\bm{a}=(a_1,a_2)\T=(1,100)\T$, and $\bm{b}=(b_1,b_2)\T=(0,0)\T$ (left). $1\,000$ generated vectors of random variates from this copula (right). The lower and upper tail-dependence coefficients are given by $\lambda_l=0$ and $\lambda_u=1$, respectively.} 
		 	\label{fig2}
		\end{figure}
	  \end{example}
	\section{An application to the pricing of first-to-default contracts}\label{sec.app}
		In order to illustrate how a joint model with Sibuya dependence structure can be applied, we consider a financial application, consisting in the valuation of contracts known as \textit{first-to-default} (FTD) swaps. The key element of this contract is a basket, i.e. a pool of $d>1$ entities. The first party, named \textit{protection seller} agrees to pay the contractual counterparty known as \textit{protection buyer} a fraction of the contract notional if the first default in the basket happens before the maturity of the contract. The fraction, named as \textit{loss given default}, is defined by one minus the recovery rate of the defaulted entity. In order to enjoy that protection, the \textit{protection buyer} is willing to pay the \textit{protection seller} a given premium at prespecified payment dates (usually quarterly IMM), up to the maturity of the deal or the first default time, whichever comes first. After the first default in the basket or at the maturity, the latest, the contract stops. 
		\par
		Mathematically, risk-neutral pricing theory proves that the above two payment legs are, per unit notional, equal to $\IE[\exp(-r\tau_{(1)})(1-R_{(1)})\I_{\{\tau_{(1)}\le T\}}]$ (the \textit{default leg}) and $s\IE[\int_0^T\exp(-rt)\I_{\{\tau_{(1)}>t\}}\,dt]$ (the \textit{premium leg}; for simplicity, we assume continuous premium payments), where $\tau_{(1)}:=\min\{\tau_1,\dots,\tau_d\}$ with distribution function $F_{\tau_{(1)}}$, $R_i$ denotes the recovery rate of the $i$th entity (with index ``(1)'' denoting the recovery rate of the first defaulting entity in the basket), $s$ is the agreed \textit{spread} (running annual premium), $r$ is the (assumed constant) risk-free interest rate used for discounting the cashflows, and $T$ is the contract's maturity in years. A basic calculation leads to the present value of such a contract from the protection buyer's perspective, given by
		\begin{align*}
			\sideset{}{_T}\FTD&=\IE\Bigl[\exp(-r\tau_{(1)})(1-R_{(1)})\I_{\{\tau_{(1)}\le T\}}\Bigr]-s\IE\Bigl[\int_0^T\exp(-rt)\I_{\{\tau_{(1)}>t\}}\,dt\Bigr]\\
			&=(1-R_{(1)})\int_0^T\exp(-rt)\,dF_{\tau_{(1)}}(t)-s\int_0^T\exp(-rt)(1-F_{\tau_{(1)}}(t))\,dt.
		\end{align*}
		\par
		Another product, known as credit default swaps (``CDS'') allows one to derive risk-neutral default probability curves from market prices. This provides us with $d$ univariate default distributions $1-\tilde{S}_i(t)$, $t\in[0,\infty]$, $i\in\{1,\dots,d\}$. However, this is not enough to infer $F_{\tau_{(1)}}$, the distribution of the first default time $\tau_{(1)}$. The survival distribution of $\tau_{(1)}$ is given by the diagonal of the joint survival distribution, which we cannot infer solely based on the $d$ univariate statistics at our disposal. In order to fill that gap, one needs to assume a dependence structure between the default times, while preserving the (market implied) individual survival probabilities. Here we achieve this with the help of our derived Sibuya copula class. In the following, we shall see how one can derive the FTD default distribution. Interestingly, in our framework, the later will be shown to admit a closed form expression; this is remarkable noting that this is not the case even for the ``simple'' Gaussian copula model. Since the purpose of this paper is not to discuss how the copula parameters can be calibrated and in order to avoid entering the technical details related to the application, we assume them as given (e.g., by an expert).
		\par
		For simplicity, assume the individual default times have the same survival function $\tilde{S}(t)=\exp(-\tilde{\lambda}t)$ calibrated to CDS market quotes. Technically, this corresponds either to a one-point or to a flat CDS spread curve. For the joint dependence structure of the vector $\bm{\tau}$ of default times, we use the Sibuya copula $C$ as given in (\ref{C.A}). Its diagonal is given in (\ref{C.A.diag}). Applying Integration by Parts, $\sideset{}{_T}\FTD$ can be computed as
		\begin{align*}
			\sideset{}{_T}\FTD&=(1-R_{(1)})\bigl(1-\exp(-rT)C(\tilde{S}(T),\dots,\tilde{S}(T))\bigr)\\
			&\phantom{={}}-((1-R_{(1)})r+s)\int_0^TC(\tilde{S}(t),\dots,\tilde{S}(t))\exp(-rt)\,dt,
		\end{align*}
		so that the \textit{fair spread}, obtained by solving $\sideset{}{_T}\FTD=0$ with respect to $s=:s^\ast$, is given by 	
		\begin{align*}
			s^\ast=(1-R_{(1)})\biggl(\frac{1-e^{-rT}C(\tilde{S}(T),\dots,\tilde{S}(T))}{\int_0^TC(\tilde{S}(t),\dots,\tilde{S}(t))\exp(-rt)\,dt}-r\biggr).
		\end{align*}
		Writing $\beta:=d-\lambda(1-e^{-H})\sum_{i=1}^{d}(1-e^{-H(d-i)})/\lambda_{[i]}$ for the exponent appearing in (\ref{C.A.diag}), we have $C(\tilde{S}(t),\dots,\tilde{S}(t))=\exp(-\beta\tilde{\lambda}t)$, so that
		\begin{align}
			s^\ast=(1-R_{(1)})\beta\tilde{\lambda}.\label{s.star}
		\end{align}
		\par
		Now let us consider the attainable FTD spreads. The boundary of the $(H,\lambda)$-space, i.e., $H=0$ or $\lambda=0$, corresponds to independence, see (\ref{C.A}). In this case $\beta$ takes its largest value ($\beta=d$), so does the spread in (\ref{s.star}). The lowest attainable FTD spread is obtained for the smallest value of $\beta$ ($\beta=1$), which, in turn, is obtained by letting $H\uparrow\infty$ and $\lambda\uparrow\infty$. The corresponding limiting copula is seen to be the upper Fr\'{e}chet bound $C(\bm{u})=M(\bm{u})$. This corresponds to maximum correlation where all entities default simultaneously when the first jump of the process $(J_t)_{t\in[0,\infty)}$ takes place.
		\par
		An example of FTD spread level curves is shown in Figure \ref{fig3}.
		\begin{figure}[htbp]
	   	\centering
	   	\includegraphics[width=0.8\textwidth]{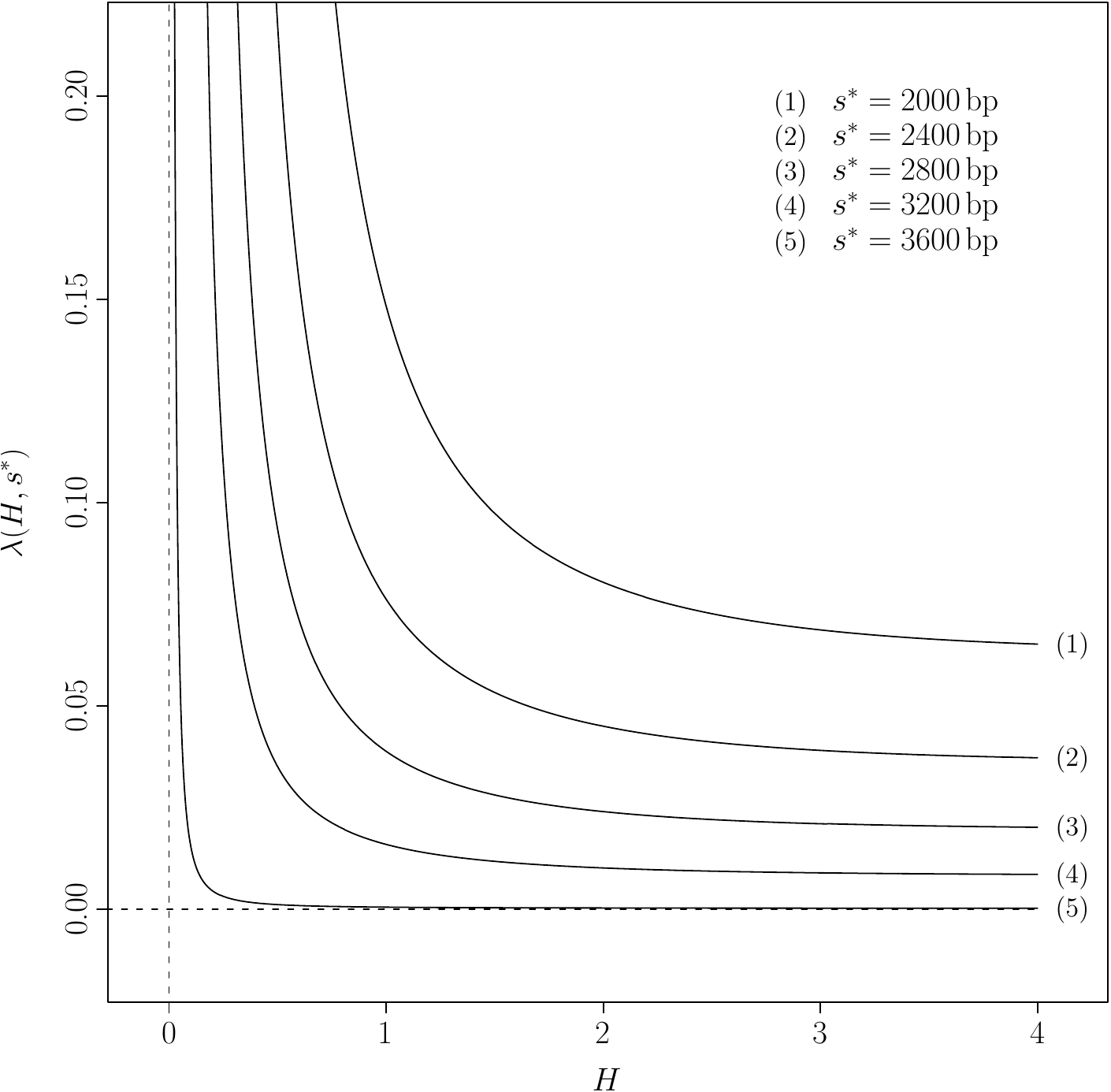}%
	   	\setcapwidth{0.8\textwidth}%
	   	\caption{We consider a basket consisting of five entities with individual default intensities $\tilde{\lambda}_i$, $i\in\{1,\dots,5\}$, given by 0.05, 0.133, 0.133, 0.142, and 0.145, respectively. Further, we take $\tilde{\lambda}$ to be the mean and $\mu_i$ to be the minimum of these values, i.e., $\tilde{\lambda}=0.1206$, $\mu_i=0.05$, $i\in\{1,\dots,d\}$. We also take $R=40\%$. The dotted lines indicate the boundary of the $(H,\lambda)$-space. Each point on the remaining curves represents the corresponding FTD spread level $s^\ast$ for the associated parameter values $H$ and $\lambda(H,s^\ast)$. The level curves are simple, and an optimization on such a space is nicely behaved.} 
	   	\label{fig3}
	   \end{figure}	
		\begin{remark}
			The joint default model which led to the Sibuya copula is built on a general default model, assuming stochastic default intensities of some entities. Those should not be confused with the deterministic intensities of the FTD basket constituents, bootstrapped from the CDS market. In other words, there is no obvious link between the functions $S^-_i$, $i\in\{1,\dots,d\}$, involved in the copula definition (and solely used to build the dependence function) and the univariate survival functions of the constituents of the basket, which will be plugged in the copula in order for the joint survival function to have the market-implied margins. However, in the particular case where one sets the model parameters $S^-_i$, $i\in\{1,\dots,d\}$, such that the functions $S_i$ are equal to the (CDS-based) survival probability curves $\tilde{S}_i$, $i\in\{1,\dots,d\}$, then the link becomes obvious, namely each realization of the stochastic survival process $p_i(t)$ can be seen as a possible default probability curve for the $i$th entity, to which some probability is assigned, and such that $\IE[p_i(t)]=\tilde{S}_i(t)$. Thus, the intuitive framework based on random intensities translates to the basket constituents only in this case.
		\end{remark}
	\section{Generalizations}\label{sec.gen}
		A key feature of our model is that it can be easily extended without affecting its tractability. In this section we briefly investigate some generalizations of the copula construction via the default model specified in (\ref{p}), (\ref{X}), and (\ref{tau}) and show that it remains perfectly workable. Although a lot of such extensions are possible (for instance, the scalar $H$ could be made name-specific so that the jumps are rescaled in a name-specific fashion, or replacing the jump process by a sum of jump processes each with possibly different scaling coefficients, etc.), we restrict ourselves to analyze two possible generalizations. First, we replace the single jump process $(J_t)_{t\in[0,\infty)}$ by a hierarchy of jump processes $(J_{j,t})_{t\in[0,\infty)}$, $j\in\{1,\dots,J\}$. Then, we address the case of dependent trigger variables $\bm{U}$.
		\subsection{Generalization to hierarchical jump processes}\label{sec.hier.jump}
		The jump process $(J_t)_{t\in[0,\infty)}$ is fatal in the sense that it hits all components of the default model simultaneously. In practical applications, it might be the case that only certain subgroups of components get hit. Such a hierarchical or sectorial behavior may be modeled via (possibly dependent) jump processes $(J_{j,t})_{t\in[0,\infty)}$ for $j\in\{1,\dots,J\}$, where $J$ is the number of sectors or subgroups. These subgroups often arise naturally from the application considered, e.g., by given industry sectors, macroeconomic effects, geographical regions, political decisions, or consumer trends. A default model incorporating such hierarchies can be constructed with the stochastic processes $(X_{i,t})_{t\in[0,\infty)}$, $i\in\{1,\dots,d\}$, in (\ref{X}) being replaced by
			\begin{align*}
		    	X_{ji,t}:=M_i(t)+J_{j,t},\ t\in[0,\infty),\ j\in\{1,\dots,J\},\ i\in\{1,\dots,d_j\}.
			\end{align*}
			The corresponding individual survival processes are then given by $p_{ji}(t):=\exp(-X_{ji,t})$, $t\in[0,\infty]$, and the default time of entity $i$ in sector $j$ by $\tau_{ji}:=\inf\{t\ge0:p_{ji}(t)\le U_{ji}\}$, where $U_{ji}\sim\U[0,1]$, independent for all $j\in\{1,\dots,J\}$, $i\in\{1,\dots,d_j\}$. For simplicity, we assume the jump processes $(J_{j,t})_{t\in[0,\infty)}$, $j\in\{1,\dots,J\}$, to be independent in what follows. Then the joint survival function $S$ can be derived similarly as in the proof of Theorem \ref{main.theorem}. First note that the individual survival functions $S_{ji}$ are given by $S_{ji}(t)=\exp(-M_i(t))\psi_{J_{j,t}}(1)$, $t\in[0,\infty]$, $j\in\{1,\dots,J\}$, $i\in\{1,\dots,d_j\}$. The joint survival function $S$ can then be calculated as
			\begin{align*}
				S(t_{11},\dots,t_{1d_1},\dots,t_{J1},\dots,t_{Jd_J})=\prod_{j=1}^J\prod_{i=1}^{d_j}\frac{\psi_{J_{j,t_{j(i)}}-J_{j,t_{j(i-1)}}}(d_j-i+1)}{\psi_{J_{j,t_{j(i)}}}(1)}S_{ji}(t_{ji}),
			\end{align*}
			where $t_{j(i)}$ denotes the $i$th smallest value of all components $\{t_{j1},\dots,t_{jd_j}\}$ in sector $j$. The corresponding copula $C$ is thus given by
			\begin{align*}
				C(\bm{u})=\prod_{j=1}^J\prod_{i=1}^{d_j}\frac{\psi_{J_{j,\Sinv{j\cdot\,}{j\cdot\,}{-0.5mm}_{(i)}}-J_{j,\Sinv{j\cdot\,}{j\cdot\,}{-0.5mm}_{(i-1)}}}(d_j-i+1)}{\psi_{J_{j,\Sinv{j\cdot\,}{j\cdot\,}{-0.5mm}_{(i)}}}(1)}u_{ji},
			\end{align*}
			which is a product of Sibuya copulas and therefore itself of the Sibuya type. Therefore, Sibuya copulas are able to capture such hierarchical default dependencies.
			\subsection{Generalization to dependent trigger variables}\label{sec.dep.trig}
			Now let us introduce dependence among the default triggers via $\bm{U}\sim C_{\bm{U}}$, i.e., the trigger variables $U_i$, $i\in\{1,\dots,d\}$, are dependent according to the copula $C_{\bm{U}}$. Note that this does not influence the marginal distributions $S_i$, $i\in\{1,\dots,d\}$, as given in (\ref{S.i}). Redoing the calculations as carried out in the proof of Theorem \ref{main.theorem} leads to joint survival function
			\begin{align*}
			   S(\bm{t})&=\IE\bigl[C_{\bm{U}}\bigl(\exp(-(M_1(t_1)+J_{t_1}),\dots,\exp(-(M_d(t_d)+J_{t_d}))\bigr)\bigr]\\
				&=\IE\biggl[C_{\bm{U}}\biggl(\frac{S_1(t_1)}{\psi_{J_{t_1}}(1)}\exp(-J_{t_1}),\dots,\frac{S_d(t_d)}{\psi_{J_{t_d}}(1)}\exp(-J_{t_d})\biggr)\biggr].
			\end{align*}
			The corresponding copula $C$ is therefore given by
			\begin{align*}
					C(\bm{u})=\IE\biggl[C_{\bm{U}}\biggl(\frac{\exp(-J_{S_1^-(u_1)})}{\psi_{J_{S_1^-(u_1)}}(1)}u_1,\dots,\frac{\exp(-J_{S_d^-(u_d)})}{\psi_{J_{S_d^-(u_d)}}(1)}u_d\biggr)\biggr].
			\end{align*}
			\begin{example}
				Let us assume the homogeneous case, i.e., assume that, pointwise, $M_i=M_1$ for all $i\in\{2,\dots,d\}$. This implies that, pointwise, $S_i=S_1$, $i\in\{2,\dots,d\}$. As an example where the copula $C$ is given explicitly, consider $C_{\bm{U}}(\bm{u})=\alpha M(\bm{u})+(1-\alpha)\Pi(\bm{u})$ for $\alpha\in[0,1]$, i.e., $C_{\bm{U}}(\bm{u})$ is a convex combination of the upper Fr\'{e}chet bound copula $M$ and the independence copula $\Pi$. In this case, applying Theorem \ref{main.theorem} leads to
				\begin{align*}
					S(\bm{t})&=\IE\Bigl[\alpha\min_{i}\{\exp(-(M_1(t_i)+J_{t_i}))\}+(1-\alpha)\prod_{i=1}^d\exp(-(M_1(t_i)+J_{t_i}))\Bigr]\\
						&=\alpha\IE[\min_{i}\{\exp(-(M_1(t_i)+J_{t_i}))\}]+(1-\alpha)\IE\Bigl[\prod_{i=1}^d\exp(-(M_1(t_i)+J_{t_i}))\Bigr]\\
						&=\alpha\IE[\exp(-(M_1(t_{(d)})+J_{t_{(d)}}))]\\
						&\phantom{={}}+(1-\alpha)\prod_{i=1}^d\frac{\psi_{J_{t_{(i)}}-J_{t_{(i-1)}}}(d-i+1)}{\psi_{J_{t_{(i)}}}(1)}S_1(t_i)\\
						&=\alpha S_1(t_{(d)})+(1-\alpha)\prod_{i=1}^d\frac{\psi_{J_{t_{(i)}}-J_{t_{(i-1)}}}(d-i+1)}{\psi_{J_{t_{(i)}}}(1)}S_1(t_i)\\
						&=\alpha\min_{i}S_1(t_i)+(1-\alpha)\prod_{i=1}^d\frac{\psi_{J_{t_{(i)}}-J_{t_{(i-1)}}}(d-i+1)}{\psi_{J_{t_{(i)}}}(1)}S_1(t_i)
				\end{align*}
				The copula corresponding to $S$ is therefore given by
				\begin{align*}
					C(\bm{u})&=\alpha M(\bm{u})+(1-\alpha)\prod_{i=1}^d\frac{\psi_{J_{\Sinv{1}{\cdot\,}{-1.3mm}_{(i)}}-J_{\Sinv{1}{\cdot\,}{-1.3mm}_{(i-1)}}}(d-i+1)}{\psi_{J_{\Sinv{1}{\cdot\,}{-1.3mm}_{(i)}}}(1)}u_i\\
								&=\alpha M(\bm{u})+(1-\alpha)\prod_{i=1}^d\frac{\psi_{J_{S_1^-(u_{(d-i+1)})}-J_{S_1^-(u_{(d-i+2)})}}(d-i+1)}{\psi_{J_{S_1^-(u_{(d-i+1)})}}(1)}u_i,
				\end{align*}
				where $u_{d+1}:=1$. Thus, we recognize that $C$ is a convex combination of the upper Fr\'{e}chet bound copula $M$ and the Sibuya copula as given in (\ref{C}) for the homogeneous case.
			\end{example}
	\section{Conclusion}\label{sec.concl}
		We introduced an intuitive default model which extends the classical intensity-based approach by allowing the survival process to jump downwards, i.e., to be stochastic. We then derived the survival distribution and, as a corollary, the associated copula which is proven to be of Sibuya type. For that reason, they are named \textit{Sibuya copulas}. Since the parameters of the marginal survival functions of the default times appear in the copula, Sibuya copulas allow for asymmetries. Due to the jump process in the construction, they allow for a singular component. We also showed that Sibuya copulas may be extreme-value copulas or L\'evy-frailty copulas, depending on the functional parameters chosen. From the construction principle presented, a sampling algorithm for these copulas is derived. Further, properties including positive lower orthant dependence, tail dependence, and extremal dependence are investigated. A financial application consisting in the pricing of first-to-default swaps is given, and the nice-looking related expressions further emphasize again the interesting tractability of the model. This tractability most likely results from the relatively simple form of the integrated intensity process (although corresponding to a quite general setup) combined to the Sibuya property. Finally, we showed that the dependence model easily extends in various ways, and, as an illustration, explicitly derived the construction of Sibuya copulas for two of them. 
	\appendix
	\section{Proofs}
	  \subsection{Proof of Lemma \ref{lemma}}\label{app.lemma}
	  \begin{proof}
	    For Part \ref{p.1}, $\sum_{k=1}^d(d-k+1)a_{(k-1)}=\sum_{k=0}^{d-1}(d-k)a_{(k)}=\sum_{k=1}^{d-1}(d-k)a_{(k)}=\sum_{k=1}^{d}(d-k)a_{(k)}$ implies
	    \begin{align*}
	      \sum_{k=1}^d(d-k+1)(a_{(k)}-a_{(k-1)})&=\sum_{k=1}^d(d-k+1)a_{(k)}-\sum_{k=1}^d(d-k+1)a_{(k-1)}\\
	      &=\sum_{k=1}^d(d-k+1)a_{(k)}-\sum_{k=1}^d(d-k)a_{(k)}=\sum_{k=1}^da_{(k)}\\
	&=\sum_{k=1}^da_k.
	    \end{align*}
	    For Part \ref{p.2}, $\sum_{k=1}^d(a_{(k)}-a_{(k-1)})=a_{(d)}$ and $\sum_{k=1}^db_ka_{(k-1)}=\sum_{k=0}^{d-1}b_{k+1}a_{(k)}=\sum_{k=1}^{d-1}b_{k+1}a_{(k)}=c\sum_{k=1}^{d-1}b_{k}a_{(k)}$ imply
	    \begin{align*}
	      \sum_{k=1}^d(1-b_k)(a_{(k)}-a_{(k-1)})&=a_{(d)}-\sum_{k=1}^db_k(a_{(k)}-a_{(k-1)})\\
	      &=a_{(d)}-\sum_{k=1}^db_ka_{(k)}+\sum_{k=1}^db_ka_{(k-1)}\\
	      &=a_{(d)}-b_da_{(d)}-\sum_{k=1}^{d-1}b_ka_{(k)}+c\sum_{k=1}^{d-1}b_{k}a_{(k)}\\
	      &=(1-b_{d})a_{(d)}+(c-1)\sum_{k=1}^{d-1}b_ka_{(k)}\\
			&=(1-cb_{d})a_{(d)}+(c-1)\sum_{k=1}^{d}b_ka_{(k)}.\qedhere
	    \end{align*}
	  \end{proof}
	  \subsection{Proof of Theorem \ref{main.theorem}}\label{app.main}
		\begin{proof}
			First consider the $i$th marginal survival function $S_i(t)=\IP(\tau_i>t)$, $t\in[0,\infty]$. By conditioning on $J_t$, it can be computed via
			\begin{align}
			    S_i(t)&=\IP(p_i(t)\ge U_i)=\IP(\exp(-(M_i(t)+J_t))\ge U_i)\notag\\
						 &=\IE[\IP(\exp(-(M_i(t)+J_t))\ge U_i\,\vert\,J_t)]=\IE[\exp(-(M_i(t)+J_t))]\notag\\
						 &=\exp(-M_i(t))\IE[\exp(-J_t)]=\exp(-M_i(t))\psi_{J_t}(1).\label{S.i}
			\end{align}
			Given this result and Lemma \ref{lemma} \ref{p.1} with $a_i:=J_{t_i}$, $i\in\{1,\dots,d\}$, the joint survival function of the default times can be computed via
			\begin{align*}
				S(\bm{t})&=\IP(p_i(t_i)\ge U_i,\ i\in\{1,\dots,d\})\\
				&=\IP(\exp(-(M_i(t_i)+J_{t_i}))\ge U_i,\ i\in\{1,\dots,d\})\\
				&=\IE[\IP(\exp(-(M_i(t_i)+J_{t_i}))\ge U_i,\ i\in\{1,\dots,d\}\,\vert\,J_{t_i},\ i\in\{1,\dots,d\})]\\
				&=\IE\biggl[\,\prod_{i=1}^d\exp(-(M_i(t_i)+J_{t_i}))\biggr]=\IE\biggl[\exp\biggl(-\sum_{i=1}^dJ_{t_i}\biggr)\biggr]\prod_{i=1}^d\exp(-M_i(t_i))\\
				&=\IE\biggl[\exp\biggl(-\sum_{i=1}^d(d-i+1)(J_{t_{(i)}}-J_{t_{(i-1)}})\biggr)\biggr]\prod_{i=1}^d\frac{S_i(t_i)}{\psi_{J_{t_i}}(1)}\\
				&=\prod_{i=1}^d\frac{\psi_{J_{t_{(i)}}-J_{t_{(i-1)}}}(d-i+1)}{\psi_{J_{t_i}}(1)}S_i(t_i)=\prod_{i=1}^d\frac{\psi_{J_{t_{(i)}}-J_{t_{(i-1)}}}(d-i+1)}{\psi_{J_{t_{(i)}}}(1)}S_i(t_i),
			\end{align*}
			where, in the second last equality, the independence assumption between any non-overlapping increments of the jump process	is used.
		\end{proof}
		\subsection{Proof of Proposition \ref{prop.extreme}}\label{app.prop}
		\begin{proof}
	    The survival copula $\hat{C}$ corresponding to $C$ can be recovered from $C$ by the Poincar{\'e}-Sylvester sieve formula. This implies
	    \begin{align*}
	        \hat{C}(1-u,\dots,1-u)&=\sum_{I\subseteq\{1,\dots,d\}}(-1)^{\vert I\vert}C(u^{\I_I(1)},\dots,u^{\I_I(d)})\\
	        &=\sum_{I\subseteq\{1,\dots,d\}}(-1)^{\vert I\vert}\prod_{i=1}^{d}u^{\I_I(i)}(u^{-\lambda\I_{I}(\cdot\,)/\lambda_{\cdot}})_{(i)}^{(1-e^{-H})(1-e^{-H(d-i)})},
	    \end{align*}
	    where $\I_I(i)$ denotes the indicator of $i$ being in $I$, $i\in\{1,\dots,d\}$. Now consider the term $(u^{-\lambda\I_{I}(\cdot\,)/\lambda_{\cdot}})_{(i)}$. If there is no $k\in\{1,\dots,d\}$ such that $k\notin I$, i.e., if $I=\{1,\dots,d\}$, then $(u^{-\lambda\I_{I}(\cdot\,)/\lambda_{\cdot}})_{(i)}=(u^{-\lambda/\lambda_{\cdot}})_{(i)}=u^{-\lambda/\lambda_{[i]}}$. If there is precisely one $k\in\{1,\dots,d\}$ such that $k\notin I$, then $(u^{-\lambda\I_{I}(\cdot\,)/\lambda_{\cdot}})_{(i)}=u^{-\lambda/\lambda_{[i]}}$ if $i\in\{2,\dots,d\}$ and $1$ if $i=1$. If there are precisely two different $k_1,k_2$ in $\{1,\dots,d\}$ such that $k_1,k_2\notin I$, then $(u^{-\lambda\I_{I}(\cdot\,)/\lambda_{\cdot}})_{(i)}=u^{-\lambda/\lambda_{[i]}}$ if $i\in\{3,\dots,d\}$ and $1$ if $i\in\{1,2\}$. Continuing this way one obtains that 
	    \begin{align*}
	        (u^{-\lambda\I_{I}(\cdot\,)/\lambda_{\cdot}})_{(i)}=\begin{cases}
	          1,&i\in\{1,\dots,\vert I\vert\},\\
	          u^{-\lambda/\lambda_{[i]}},&i\in\{\vert I\vert+1,\dots,d\}.
	        \end{cases}
	    \end{align*}
	   This implies that
	   \begin{align*}
	       \hat{C}(1-u,\dots,1-u)&=\sum_{I\subseteq\{1,\dots,d\}}(-1)^{\vert I\vert}u^{\vert I\vert}\prod_{i=\vert I\vert+1}^{d}u^{-(\lambda/\lambda_{[i]})(1-e^{-H})(1-e^{-H(d-i)})}\\
	       &=\sum_{I\subseteq\{1,\dots,d\}}(-1)^{\vert I\vert}u^{\vert I\vert-\lambda(1-e^{-H})\sum_{\vert I\vert+1}^d(1-e^{-H(d-i)})/\lambda_{[i]}}\\
	       &=\sum_{I\subseteq\{1,\dots,d\}}(-1)^{\vert I\vert}u^{\vert I\vert-c\sum_{\vert I\vert+1}^da_i},
	   \end{align*}
	   where $c:=\lambda(1-e^{-H})$ and $a_i:=(1-e^{-H(d-i)})/\lambda_{[i]}$, $i\in\{1,\dots,d\}$. Finally, note that
	   \begin{align*}
	      C(u,\dots,u)=u^{d-c\sum_{i=1}^da_i}.
	    \end{align*}
	    With these two ingredients, we obtain
	    \begin{align*}
	        \eps_l&=\lim_{u\downarrow0}\frac{u^{d-c\sum_{i=1}^da_i}}{1-\hspace{-2mm}\sum\limits_{I\subseteq\{1,\dots,d\}}\hspace{-2mm}(-1)^{\vert I\vert}u^{\vert I\vert-c\sum_{i=\vert I\vert+1}^da_i}},\\
	 \eps_u&=\lim_{u\uparrow1}\frac{\sum\limits_{I\subseteq\{1,\dots,d\}}(-1)^{\vert I\vert}u^{\vert I\vert-c\sum_{i=\vert I\vert+1}^da_i}}{1-u^{d-c\sum_{i=1}^da_i}}.
	    \end{align*}
	    The first statement now directly follows from the formula for $\eps_l$. For the second statement, apply l'H\^opital's Rule.
	  \end{proof}
\bibliographystyle{plainnat}
\bibliography{./mybibliography}
\end{document}